\documentclass[reqno]{amsart}

\usepackage{amsmath}
\usepackage{amsfonts}

\newtheorem{thm}{Theorem}[section]
\newtheorem{lem}{Lemma}[section]

\theoremstyle{definition}
\newtheorem{defin}{Definition}[section]

\begin{document}
\title{The Reciprocal Pascal Matrix}
\vspace{.5in}

\author[T. M. Richardson]{Thomas M. Richardson \\
Ada, MI  49301}

\email{ribby@umich.edu}

\begin{abstract}
The reciprocal Pascal matrix is the
Hadamard inverse of the symmetric Pascal matrix.
We show that the ordinary matrix inverse of the
reciprocal Pascal matrix has integer elements.
The proof uses two factorizations of the
matrix of super Catalan numbers.
\end{abstract}

\maketitle

\section{Background}
\begin{defin}
The reciprocal Pascal matrix is the matrix $R$ whose
$(i,j)$-element is 
\begin{equation}
R_{i,j} = \frac{1}{\binom{i+j}{i}}, 
\end{equation}
for $0 \le i,j$.
\end{defin}
The $n \times n$ reciprocal Pascal matrix is the initial
segment of this matrix, $R_{i,j}$ for $0 \le i,j < n$.
(All matrices in this paper are indexed starting at $0$.)
The determinant of the inverse of the $n \times n$ reciprocal
Pascal matrix is sequence A060739 in the 
Online Encyclopedia of Integer Sequences \cite{OEIS}.
Katz asked about a formula for the determinant of this matrix in 2001\cite{BobKatz}. 
Israel stated the formula for the determinant of
the reciprocal Pascal matrix, and conjectured that 
the inverse of the reciprocal Pascal matrix is
an integer matrix \cite{Israel}. 
The author suggested a proof that the inverse of 
the reciprocal Pascal matrix is an integer matrix 
in 2003 \cite{RCP}, using an ad hoc factorization of
the reciprocal Pascal matrix.

This paper gives a proof that the inverse of the $n \times n$ reciprocal Pascal matrix is an integer matrix,
using the matrix of super Catalan numbers \cite{GESSEL}. 
The new proof is simpler than the previous proof, 
and it has the advantage of showing the formula for the determinant 
of the reciprocal Pascal matrix.

\section{The Super Catalan Matrix}
\begin{defin}The super Catalan matrix is defined by
\begin{equation}
S_{m,n} = 
\frac{(2m)!(2n)!}{m!n!(m+n)!}.
\end{equation}
\end{defin}
Multiplying both numerator and denominator by $ m!n! $ leads to the equivalent definition
\begin{equation}
S_{m,n} = 
\frac{\binom{2m}{m}\binom{2n}{n}}{\binom{m+n}{m}}
\end{equation}
We restate this equation as a matrix factorization in the next lemma.
This lemma will give one of the two factorizations of the super Catalan matrix
that will be used to prove our results.
To state the lemma we need the definition of matrix $G$.
\begin{defin}
\label{GMATRIX}
Define the diagonal matrix $G$ by
\begin{equation}
G(m,m) =
\binom{2m}{m}.
\end{equation}
\end{defin}
\begin{lem}
\label{SCATFACTOR}
The super Catalan matrix has the factorization
\begin{equation}
S = GRG,
\end{equation}
where $R$ is the reciprocal Pascal matrix and $G$ is from definition \ref{GMATRIX}.
\end{lem}

The von Szily identity, equation (29) of \cite{GESSEL}, is
\begin{equation}
S_{m,n}=\sum_{k}(-1)^k\binom{2m}{m+k}\binom{2n}{n-k}.
\end{equation}
Although the summation includes both positive and negative values of $k$,
we can use the symmetry of the terms for positive and negative $k$ to derive
the $LDL^{T}$ decompostion of $S$.
\begin{lem}
\label{SCATLDL}
The super Catalan matrix $S$ has the factorization $S=LDL^{T}$ where 
\begin{equation}
L_{m,k} = \binom{2m}{m+k}
\end{equation}
and $D$ is the diagonal matrix with
$D_{0,0} = 1$ and $D_{m,m} = (-1)^m 2$ 
for $m>0$.
\end{lem}

\begin{proof}
Since $\binom{2n}{n-k} = \binom{2n}{n+k}$, the von Szily identity is equivalent to
\begin{equation}
S_{m,n}=\binom{2m}{m}\binom{2n}{n}+2\sum_{k>0}(-1)^k\binom{2m}{m+k}\binom{2n}{n+k}.
\end{equation}
This equation is equivalent to $S=LDL^{T}.$
\end{proof}

We remark that $L$ is lower triangular with ones on the diagonal, so it has determinant $1$,
and its inverse is an integer matrix.
Also, column $0$ of $L^{-1}$ is equal to the diagonal of $D$.
The matrices $L$ and $L^{-1}$ are sequences A094527 and A110162, respectively.

\section{Inverting the Reciprocal Pascal Matrix}
In this section we prove results about the determinant and elements of the inverse of the
$n \times n$ reciprocal Pascal matrix; we assume all matrices in this section are $n \times n$.
Equating the $GRG$ and $LDL^{T}$ factorizations of the super Catalan matrix S, and isolating $R$ gives the equation
\begin{equation}
R = G^{-1}LDL^{T}G^{-1}. 
\end{equation}
Inverting both sides we have
\begin{equation}
\label{RINVERSE}
R^{-1} = G(L^{T})^{-1}D^{-1}L^{-1}G.
\end{equation}

The determinant of the $n \times n$
reciprocal Pascal matrix is clear from the factorization in equation \eqref{RINVERSE}.
\begin{thm}
The determinant of the $n \times n$ reciprocal Pascal matrix is
\begin{equation}
\det(R^{-1})=\frac{(-1)^{n(n+1)/2}}{2^{n-1}}\prod_{m=0}^{n-1}\binom{2m}{m}^2.
\end{equation}
\end{thm}
\begin{proof}
The determinants of $D^{-1}$, $L^{-1}$, and $G$ are $\frac{(-1)^{n(n+1)/2}}{2^{n-1}}$, $1$, and 
$\prod_{m=0}^{n-1}\binom{2m}{m}$, respectively.
\end{proof}

Next we show that the elements of the inverse of the $n \times n$ reciprocal Pascal matrix are integers.
Since $L^{-1}$ and $G$ are integer matrices, 
so we only have to account for the values $2$ in the denominators of the elements of $D^{-1}$.

It will be helpful to have this lemma about the first column of $L^{-1}$.

\begin{lem}
\label{LINVERSE}
The elements of the first column of $L^{-1}$ satisfy
$L^{-1}_{0,0} = 1$ and $L^{-1}_{i,0}$ is even for $i>0$.
\end{lem}
\begin{proof}
Consider $L$ as a block matrix with four blocks $A_{i,j}$, where $A_{0,0} = L_{0,0}$ and $A_{1,0}$
is everthing else in column $0$. From the definition of $L$, the elements of $A_{1,0}$ are the central
binomial coefficients $\binom{2m}{m}$ for $m>0$. Since these coefficients are even, so are all the
the elements of the corresponding block of $L^{-1}$.
\end{proof}

\begin{thm}
The inverse of the $n \times n$ reciprocal Pascal matrix is an integer matrix.
\end{thm}
\begin{proof}
The only non-integers in the factors of equation \eqref{RINVERSE} are the values $2$ in the denominator of $D^{-1}$.
Since $G$ is diagonal, and $G_{m,m}=\binom{2m}{m}$ is even for $m>0$,  
it follows that $R^{-1}_{i,j}$ is an integer if either $i>0$ or $j>0$. 
The remaining element to consider, $R^{-1}_{0,0}$, is an integer by Lemma \ref{LINVERSE}, as
$\displaystyle 
R^{-1}_{0,0}=1+\sum_{i=1}^{n-1} (L^{-1}_{i,0})^2 D^{-1}_{i,i}$.
\end{proof}
 
Sequences: A000984, A068555, A007318, A060739, A094527, A110162

AMS Classification Numbers: 11B39, 11B65, 15A09.
\end{document}